\documentclass[10pt]{article}
\usepackage{amsmath,amssymb}
\usepackage[T1]{fontenc}
\usepackage{lmodern,textcomp}
\usepackage{color}
\usepackage{graphicx}
\usepackage{amsthm}
\usepackage[all]{xy}
\usepackage{verbatim}
\usepackage{caption}

\usepackage{pgf,tikz-cd}
\usetikzlibrary{arrows}

\def\P{\mathbb{P}}
\def\O{\mathcal{O}}
\def\A{\mathcal{A}}

\def\M{{\mathcal{M}}}
\def\oR{{\mathcal{R}}}

\def\Nm{{\rm Nm}}

\def\im{{\rm Im}}

\begin{document}

\theoremstyle{definition}
\newtheorem{mydef}{Definition}[section]
\newtheorem{mynot}[mydef]{Notation}
\newtheorem*{example}{Example}
\newtheorem*{remark}{Remark}
\theoremstyle{plain}
\newtheorem{myth}{Theorem}
\newtheorem{mypr}[mydef]{Proposition}
\newtheorem{mylem}[mydef]{Lemma}
\newtheorem{mycor}[mydef]{Corollary}

\title{On the principally polarized abelian varieties with $m$-minimal curves}
\author{Shin-Yao Jow \and Adrien Sauvaget \and Hacen Zelaci}
\date{\today}
\maketitle

\abstract{In this paper, we study principally polarized abelian varieties $X$ of dimension $g$ with a curve $\nu:C\to X$ such that the class of $C$ is $m$ times the minimal class. In \cite{Wel}, Welters introduced the formalism of complementary pairs to handle this problem in the case $m=2$. We generalize the results of Welters and construct families of principally polarized abelian varieties for any $m$ and compute the dimension of the locus of these abelian varieties.}

\tableofcontents

\section{Introduction}

Let $g>0$, we will denote by $\A_g$ the \textit{moduli space of complex principally polarized abelian varieties} (ppav). In order to study the geometry of $\A_g$, a classical approach is to construct a stratification using moduli spaces of curves. The first stratum is the locus of Jacobians, which is isomorphic to the moduli space of curves by the Torelli theorem, which is contained in the closure of the Prym locus obtained from the space of unramified coverings of degree two. For example  the space $\A_5$ is uniformized by the space of degree two unramified coverings of a genus $6$ curve (see \cite{Don} and \cite{Bea}). However, the Jacobian locus has dimension $3g-3$ and the Prym locus has dimension $3g$ while the space $\A_g$ has dimension $\frac{g(g+1)}{2}$. Thus, the Jacobian and Prym loci will provide a high codimensions subspaces when $g$ becomes big.

Other constructions using spaces of coverings with more general groups of monodromy have been extensively studied (see \cite{Kan} or \cite{Kan1}); recently this has lead to the uniformization of $\A_6$ (see \cite{Unif6}).

In order to stratify the moduli space $\A_g$, one can generalize the Jacobian and Prym loci. A Prym-Tjurin variety of exponent $m$ is a principally polarized abelian variety $(X,\Theta)$ of dimension $g$ together with a map $\nu:C\to X$ such that
$$\nu_*([C])=m.\frac{\Theta^{g-1}}{(g-1)!}.$$
We will often denote by $z=\frac{\Theta^{g-1}}{(g-1)!}$ the \textit{minimal class}. Note that this  class is not divisible. 

The important point to remark is that, by a theorem of Welters, any ppav is a Prym-Tjurin variety for a certain exponent $m$. Indeed, let $X$ be ppav of dimension $g$, consider a canonical embedding of $X\to \P^{N}$. We can intersect $X$ with $g-1$ sections of the line bundle $\O(k)$. For $k$ sufficiently large, this intersection is not empty. Thus assuming that the intersection is transverse we get a curve in $X$ which is $m$ times the minimal class for a large $m$. Therefore, the loci of Prym-Tjurin varieties of exponent $m$ provide a stratification of the spaces $\A_g$. However, it is difficult to  produce families of Prym-Tjurin varieties of a general $m$.

\paragraph{Complementary pairs.} In \cite{Wel}, Welters proposed a slightly more general version of the Prym-Tjurin varieties. A ppav will be called $m$-minimal if there exists a curve $\nu:C\to X$ such that $\nu_*([C])$ is $m$ times the minimal class; we will say that the curve $C$ is $m$-minimal. The interesting point is that we can reformulate this property using the theory of \textit{complementary pairs}. If $X$ is ppav, we will denote by $\lambda_X:X\to \hat{X}$ the morphism induced by the polarization.  Let $N$ be a smooth curve. Then the following two sets of data are equivalent:
 \begin{enumerate}
 \item A ppav $(X,\Theta)$ and a morphism $\nu:N\to X$ such that $\nu_*([N])=m\cdot z$.
 \item An abelian subvariety $B$ of $JN$ and a finite subgroup $H\subset JN$, such that, if  $\lambda_B$ is the induced polarization  on $B$ by the principal polarization $\lambda_N$ of $JN$ then:
 \begin{enumerate}
 \item $\ker{\lambda_B}\subset H\subset (JN)_m$ ($m-$torsion points of $JN$).
 \item $H$ is maximal totally isotropic subgroup of $(JN)_m$.
 \end{enumerate}
 \end{enumerate}

 We describe a little bit more the construction from 2 to 1. Let $N$ be a smooth curve and $B$ a subvariety satisfying the conditions of 2. Let $A\subset JN$ be obtained by the dual exact sequence:
 \begin{center}
$\xymatrix{
0\ar[r] & \widehat{JN/B} \ar[r] & \widehat{JN}=JN \ar[r] & \widehat{B} \ar[r] & 0 \\
0\ar[r] & A \ar[r] \ar[u]_{\simeq}& JN \ar[r]\ar[u]_{\simeq}^{\lambda_N} & JN/A \ar[u]_{\simeq} \ar[r] & 0.
}$
\end{center}
Thus we get a couple $(A,B)$ of abelian subvarieties of $JN$  satisfying  $A\cap B=\ker \lambda_B=\ker\lambda_A$. This pair is called a {\em complementary pair}. All the statements that follow will be valid if we exchange the roles of $A$ and $B$.

 We denote by $\mu_B:\hat{B}\to B$ the polarization dual to $\lambda_B$. We have  $\lambda_B\circ \mu_B=[m]_B$ and the following sequence is exact $$0\to \ker \lambda_B\to B_m \to \ker \mu_B\to 0.$$
 In particular, in the condition $2$, we can replace the datum of $H$ a maximal totally isotropic (m.t.i.) subgroup of $B_m$ containing $\ker(\lambda_B)$ by the choice of a m.t.i. subgroup $K$ of  $\ker{\mu_B}$.

We denote by $\tau$ the natural isogeny from $A\times B \to JN$ which maps $(a,b)$ to $a+b$. The kernel of $\tau$ is  given by $\{(x,-x),x \in A\cap B\}$. We define $j:JN\to JN$ the unique morphism which makes  the following diagram commutative
 \begin{center}
$\xymatrix{
JN \ar[r]^j  & JN \\
A\times B \ar[u]^\tau \ar[r]_{\left(\begin{smallmatrix} 1 & 0 \\ 0 & 1-m \end{smallmatrix}\right)} & A\times B \ar[u]_\tau.
}$
\end{center}
Then $A=\im(j+m-1)=\ker(1-j)^0$ and $B=\im(1-j)=\ker(j+m-1)^0$ (we can permute the roles of $A$ and $B$ by taking $j'=m-2-j$). The morphism $j$ satisfies the classical Prym-Tjurin property $(j-1)(j-m+1)=0$.

With this set-up, the variety $X$ will be defined as $\widehat{B}/K\simeq (JN/A)/K$ and the map $\nu$ from $N$ to $X$ is given by the composition of maps
$$N\to JN\to JN/A\to (JN/A)/K.$$ 
We denote by $u$ the map from $JN$ to $X$ and by $u^t=\lambda_N^{-1} \circ u^* \circ  \lambda_X$. Then we have $uu^t=[m]$. The class of $N$ in $JN$ is the minimal class of $JN$, therefore  the class $\nu_*([N])$ is equal to $m$ times the minimal class.

\paragraph{Case $m=2$.} We have three different constructions in the case $m=2$.
\begin{enumerate}
\item Quotients of Jacobians. We take $B=JN$ and $H$ is any maximal isotropic subgroup of $JN_2$.
\item Quotients of Prym varieties. Let $\pi: N\to N_0$ be a double covering.  We denote by $\pi^*: JN_0\to JN$ and by $\Nm_\pi$ (or simply $\Nm$ if the context is clear) the norm map associated to $\pi$. Moreover, the covering comes with an involution $\iota:N\to N$. Then by taking  $B=\ker(\Nm)_0$ (the neutral component) the Prym variety attached to $\pi$, $A=\pi^*(JN_0)$ and $j=\iota^*$ we get an $m$-minimal curve by choosing a m.t.i. subgroup of $\ker(\mu_B)$.
\item Quotient of pull-back of Jacobians. This construction is obtained by exchanging the roles of $A$ and $B$. It is knowing that $\ker(\mu_B)$ is non-trivial if and only if $\pi$ is unramified.
\end{enumerate}
Welters proved that any $2$-minimal ppav is  of one of the three above types.

\paragraph{Statement of the results.} The current paper has two purposes. The first one is to compute the dimension of loci defined by Welters in the case $m=2$ inside the moduli space of ppav. The expected dimensions for quotients of Jacobians is $3g-3$. For the Prym and inverse Prym varieties, the expected dimension is the dimension of the Hurwitz space classifying the covering $N\to N_0$.
\begin{mypr}
The dimensions of all the families of ~$2$-minimial ppav are the expected ones. In particular the dimension of the locus of ~$2$-minimal ppav in $\A_g$ is $3g$ (obtained from the Prym varieties of unramified coverings).
\end{mypr}

The second purpose is to study the generalization of these three families of $m$-minimal ppav for $m \geq 3$. Two important questions about $m$-minimal ppav arise:
\begin{enumerate}
\item Compute the dimension of the locus of $m$-minimal varieties in $\A_g$.
\item Fixing $m$ and $g$, what are the bound on the geometric genus of $N$ if we assume that the image of $N$ is irreducible? For $m=1$, Matsusaka criterion (see \cite{mat}) imposes that $g(N)=g$. For $m=2$, the classification of $2$-minimal curves (see \cite{Wel}) implies that $g\leq g(N) \leq 2g+1$ and that these bounds can be realized. For general $m$, the lowest bound is $g$ (and we will provide examples). However the existence (and the value) of an upper bound is still open even for $m=3$. We expect this upper bound to take the form $g^+(g,m)=mg+P(m)$ where $P$ is a polynomial.
\end{enumerate}

We will prove the following
\begin{myth}
For any $m$, there exists $(3g-3)$-dimensional families of $m$-minimal ppav of dimension $g$ such that the geometric genus of the curve is given by $g$ (see Section \ref{sec:quot}) and $mg-m+1$ (see Section \ref{sec:invprym}).
\end{myth}
The dimension of these families should be far from the dimension of the locus of $m$-minimal ppav. However, these families provide an inequality $g^+(g,m)\geq mg-m+1$ (which is close to the expected one for big values of $g$).

\paragraph{Acknowledgement} We would like to thank Juan Carlos Naranjo for having proposed this subject and for his many helpful advices. We are very grateful to V\'ictor Gonz\'alez-Alonso for his help all along the work. Finally we would like to thank Alfio Ragusa,  Francesco Russo and Giuseppe Zappal\`a, the organizers of the school PRAGMATIC, for having created the environment  for the present work.

\section{Quotient of Jacobians}\label{sec:quot}
Let $\mathcal M_g$ be the moduli space of smooth curves of genus $g$. The {\em Torelli map} associates to an element $[C]\in\mathcal M_g$ the class of the Jacobian $JC$ in $\mathcal A_g$. The Torelli map is injective. Moreover it is an immersion outside the locus of hyperelliptic curves.

Let $C$ be a smooth curve of genus $g$. Let $K\subset JC_m$ be a m.t.i subgroup of $JC$ with respect to the Riemann bilinear form inside the $m$ torsion point of $JC$. Then the quotient $JC/K$ is a ppav and the composition $C\hookrightarrow JC \to X$ gives an $m$-minimal curve of $X$.

\begin{mylem} For all $[C]\in \mathcal M_g$ the map $f:C\rightarrow X$ is generically injective.
\end{mylem}
\begin{proof} If we assume that $f:C\to X$ is not generically injective. Denote by $\tilde{C}$ the reduced image of $f$ in $X$. Then $f$ is a covering (possibly ramified) on $\tilde{C}$,  hence the genus of the normalisation of $\tilde{C}$ is strictly smaller than $g_C$.
On the other hand, by definition, $JC\longrightarrow X$ is an isogeny, it follows that the curve $\tilde{C}$ generates $X$ (as a group). So the genus of the normalisation of $\tilde{C}$ is at least $dim(X)$, a contradiction.
\end{proof}

Let $g>0$. Let $\delta=(d_1,\ldots,d_g)$ be a polarization type. We will denote by $A_g^\delta$ the {\em moduli of abelian varieties with polarization of type $\delta$}. We define the following moduli space

\begin{mydef}
Let $g>0$ and $\delta$ be a polarization type. For $m>0$, we define $A_{g,m}^\delta$ to be the moduli space of pairs $(X,K)$ where $X$ is an abelian variety of dimension $g$ with polarization type $\delta$ and $K$ is a m.t.i. of the group of $m$-torsion points. 
\end{mydef}
The forgetful map $A_{g,m}^\delta \to A_{g}^\delta$ which maps $(X,K)$ to $X$ is \'etale. Thus we have $\dim(A_{g,m}^\delta)=\frac{g(g+1)}{2}$.  We have also
\begin{mypr}\label{dimquot}
Let $f: A_{g,m}^\delta \to A_g$ be the map which sends a pair $(X,K)$ to the quotient $X/K$. The map $f$ is \'etale.
\end{mypr}

Applied to the principaly polarized varieties, the dimension of the locus of $m$-minimal abelian varieties defined by a pair $(JC,K)$ is the same as the dimension of the Jacobian locus: $3g-3$.

\section{Quotient of pull-back of Jacobians}\label{sec:invprym}

Let $\pi\colon N\to N_0$ be a finite map of degree~$m$ between smooth projective curves $N$ and $N_0$. Let $A=(\ker \Nm)_0$ and $B=\pi^*JN_0$. Then $(A,B)$ is a complementary pair in $JN$ with $A\cap B\subset JN_m$. By Welters construction \cite[Proposition~1.17]{Wel}, given a maximal totally isotropic subgroup $K$ of $\ker \mu_B$, one obtains a ppav $X=B'/K$ (where $B'=JN/A\cong B^\vee$) together with a morphism $v\colon N\to X$ such that $v_*[N]=mz$, where $z=\Theta^{\dim X-1}/(\dim X-1)!$ is the minimal class. The following proposition describes those ppav $X$ obtained in this way such that the morphism $v\colon N\to X$ is birational onto its image.

\begin{mypr} \label{p:jow}
With notations as above, the following statements hold.
\begin{enumerate}
  \item Suppose that $m$ is prime. If the morphism $v\colon N \to X$ is birational onto its image, then $\pi\colon N\to N_0$ is an unramified cyclic cover. Moreover, let $\eta\in (JN_0)_m$ be the $m$-torsion line bundle on $N_0$ attached to $\pi$. Then $X\cong JN_0/[m]^{-1}\langle\eta\rangle$, where $\langle\eta\rangle$ is the subgroup  generated by $\eta$ in $JN_0$, and $[m]^{-1}\langle\eta\rangle$ is its preimage under the morphism $[m]\colon JN_0\to JN_0$.
  \item Conversely, let $m$ be any positive integer, $\eta$ be any $m$-torsion line bundle on $N_0$, and let $\pi\colon N=\mathbf{Spec}(\mathcal{O}_{N_0}\oplus \eta^{-1})\to N_0$ be the unramified cyclic cover associated to $\eta$. Then there exists a maximal totally isotropic subgroup $K$ of $\ker \mu_B$ such that $v\colon N \to X$ is birational onto its image and $X\cong JN_0/[m]^{-1}\langle\eta\rangle$.
\end{enumerate}
\end{mypr}

\begin{proof}
We have the following commutative diagram of morphisms between abelian varieties. The central vertical morphism $\overline{\Nm}$ is induced from the norm map $\Nm\colon JN\to JN_0$. All the maps with labels, as well as the quotient map $B'\to B'/K=X$, are isogenies. 
$$
\xymatrix@R=1em{
 & JN \ar[r] & JN/\ker \Nm^0 & \\
 JN_0 \ar[r]^{\pi^*} \ar[ddrr]_{[m]} & \pi^*JN_0=B \ar[r]^{\lambda_B} \ar@{^{(}->}[u]& B'\ar@{=}[u] \ar[r] \ar@/^1pc/[dddd]^{\mu_B} \ar[dd]_{\overline{\Nm}} &B'/K=X \\
 &&\\
 && JN_0\ar[dd]_{\pi^*}\\
 &&\\
 && \pi^*JN_0=B
}
$$
The central horizontal row of the diagram gives $X\cong JN_0/(\lambda_B\circ \pi^*)^{-1}(K)$. Also from the diagram we have $(\lambda_B\circ \pi^*)^{-1}(K)=[m]^{-1}\bigl(\overline{\Nm}(K)\bigr)$. Hence \[
 X\cong JN_0/[m]^{-1}\bigl(\overline{\Nm}(K)\bigr). \]
 
By \cite[Proposition~11.4.3]{BirLan}, if $m$ is prime then  \[
\ker(\pi^*\colon JN_0\to JN)=\begin{cases} \langle \eta\rangle, &\parbox{44ex}{if $\pi\colon N=\mathbf{Spec}(\mathcal{O}_{N_0}\oplus \eta^{-1})\to N_0$ is an unramified cyclic cover defined by $\eta\in (JN_0)_m$;} \\
 0, &\text{otherwise.}
\end{cases} \]
If $\ker(\pi^*\colon JN_0\to JN)=0$, then $\ker(\Nm\colon JN\to JN_0)$ is connected since $\pi^*=\Nm^\vee$. This means that the two vertical maps $\overline{\Nm}\colon B'\to JN_0$ and $\pi^*\colon JN_0\to \pi^* JN_0$ in the commutative diagram are isomorphisms, so their composition $\mu_B$ is also an isomorphism. Hence $K=\ker \mu_B=0$. It follows that $X=B'\cong JN_0$, and the map $v\colon N\to X$ is the composition of the Abel-Jacobi map $N\hookrightarrow JN$ followed by the norm map $\Nm\colon JN\to JN_0=X$. But this $v$ is clearly not birational onto its image. Hence if $m$ is prime and $v$ is birational onto its image, then $\pi\colon N\to N_0$ must be an unramified cyclic cover.
 
Let $m$ be any positive integer, $\eta$ be any $m$-torsion line bundle on $N_0$, and let $\pi\colon N=\mathbf{Spec}(\mathcal{O}_{N_0}\oplus \eta^{-1})\to N_0$ be the unramified cyclic cover associated to $\eta$. Then there is an automorphism $\sigma$ of $N$ of order~$m$ such that $N_0=N/\langle\sigma\rangle$. The kernel of the norm map $\Nm\colon JN\to JN_0$ has $m$ connected components $P_0,\ldots,P_{m-1}$, where \[
 P_\ell=\{\textstyle \mathcal{O}_N\bigl(\sum n_i(x_i-\sigma(x_i))\bigr)\mid x_i\in N,\ n_i\in\mathbb{Z},\ \sum n_i \equiv \ell\ \ (\mbox{mod}\ m)\}. \]
Hence $\ker(\overline{\Nm}\colon B'=JN/P_0\to JN_0)=\{P_0,\ldots,P_{m-1}\}$ is a cyclic group of order~$m$ generated by $P_1$. Since $\mu_B=\pi^*\circ \overline{\Nm}$, and $\ker\pi^*=\langle\eta\rangle$ is also a cyclic group of order~$m$, $\ker\mu_B\cong (\mathbb{Z}/m\mathbb{Z})^2$. More precisely, pick any $\xi\in JN$ such that $\Nm(\xi)=\eta$, and denote by $\bar{\xi}$ its image in $B'=JN/P_0$. Then an isomorphism $(\mathbb{Z}/m\mathbb{Z})^2\to \ker\mu_B$ can be defined by sending $(1,0)$ to $\bar{\xi}$ and $(0,1)$ to $P_1$. If $m$ is prime, every proper subgroup of $(\mathbb{Z}/m\mathbb{Z})^2$ is cyclic, so maximal totally isotropic subgroups of $\ker\mu_B$ are precisely those subgroups of the form $K=\langle a\bar{\xi}+bP_1 \rangle$ where $a$ and $b$ are integers with $\gcd(a,b,m)=1$. If $m$ is not necessarily prime, one can at least see that $K=\langle \bar{\xi} \rangle$ is a maximal totally isotropic subgroup of $\ker\mu_B$.

We next analyze which of these maximal totally isotropic subgroups $K$ of $\ker\mu_B$ give rise to morphisms $v\colon N\to X$ that are birational onto their images. Recall that $v\colon N \to X$ is the composition \[
 v\colon N\hookrightarrow JN\to JN/P_0=B'\to B'/K=X \]
of the Abel-Jacobi map $N\hookrightarrow JN$ followed by the quotient maps. So for two distinct points $x,y\in N$,
\begin{align*}
v(x)=v(y) &\implies \mathcal{O}_N(x-y)\in\langle P_0,\xi,P_1 \rangle\text{ in }JN\\
&\implies \mathcal{O}_{N_0}\bigl(\pi(x)-\pi(y)\bigr)=\Nm\bigl(\mathcal{O}_N(x-y)\bigr)\in\langle\eta \rangle\\
&\implies \mathcal{O}_N\bigl(\pi^{-1}(\pi(x))-\pi^{-1}(\pi(y))\bigr)=\pi^*\mathcal{O}_{N_0}\bigl(\pi(x)-\pi(y)\bigr)=0.
\end{align*}
If $v$ is not birational onto its image, then for all but finitely many points $x\in N$, there exists $y\in N$ different from $x$ such that the divisors $\pi^{-1}(\pi(x))$ and $\pi^{-1}(\pi(y))$ are linear equivalent. If $\pi(x)\ne\pi(y)$ then $\pi^{-1}(\pi(x))$ and $\pi^{-1}(\pi(y))$ generate a $g^1_m$ on $N$. This can only happen for at most finitely many $x$, for otherwise the $g^1_m$ consists of fibers of $\pi\colon N\to N_0$ and $N_0\cong \mathbb{P}^1$. Hence $\pi(x)=\pi(y)$, i.e. $y=\sigma^\ell(x)$ for some $\ell\in\mathbb{Z}/m\mathbb{Z}$. Since \[
 \mathcal{O}_N(x-y)=\mathcal{O}_N\bigl(x-\sigma^\ell(x)\bigr)=\mathcal{O}_N\Bigl(\textstyle\sum\limits_{i=0}^{\ell-1} \sigma^i(x)-\sigma\bigl(\sigma^i(x)\bigr)\Bigr)\in P_\ell, \]
the image $\overline{\mathcal{O}_N(x-y)}$ of $\mathcal{O}_N(x-y)$ in $B'=JN/P_0$ is equal to $P_\ell=\ell P_1$. It follows that $v(x)=v(y)$ in $X=B'/K$ if and only if $\ell P_1\in K$ in $B'$. Hence \[
\text{$v$ is birational onto its image} \iff \ell P_1\notin K \text{ for all }\ell\ne 0. \]
Thus if $m$ is prime and $v$ is birational onto its image, then $K=\langle a\bar{\xi}+bP_1 \rangle$ for some integers $a$ and $b$ with $m\nmid a$, and hence \[
X\cong JN_0/[m]^{-1}\bigl(\overline{\Nm}(K)\bigr)=JN_0/[m]^{-1}\langle\eta\rangle. \]
Conversely, given any integer $m>0$ and any $m$-torsion line bundle $\eta$ on $N_0$, pick any $\xi\in JN$ such that $\Nm(\xi)=\eta$, and let $K=\langle \bar{\xi}\rangle$. Then $K$ is a maximal totally isotropic subgroup of $\ker \mu_B$ such that $v\colon N \to X$ is birational onto its image and $X\cong JN_0/[m]^{-1}\langle\eta\rangle$.
\end{proof}

If $N_0$ is a curve of genus~$g$, and $\pi\colon N\to N_0$ is an unramified cyclic cover of degree~$m$, then the genus of $N$ is $mg-m+1$. So Proposition~\ref{p:jow} implies the following

\begin{mycor}
For any integers $m>0$ and $g>1$, the locus of ppav with an $m$-minimal curve of genus $mg-m+1$ is at least $3g-3$.
\end{mycor}

\begin{proof}
Let $\oR_{g,m,r}$ be the moduli space of maps $f:N\to N_0$ where:
\begin{itemize}
\item  $N$ and $N_0$ are smooth connected curves:
\item   the curve $N_0$ is of genurs $g$;
\item  the map $f$ is of degree $m$ and is ramified over a reduced divisor of degree $r$.
\end{itemize}
 The dimension of $\oR_{g,m,r}$ is $3g-3+r$. We assume in this section that $r=0$.

Let $\M_{g}^{1/m}$ be the moduli space of objects pairs $(N_0,L)$ where $N_0$ is a smooth curve of genus $g$ and $L$ is a line bundle of $N_0$ such that $L^{\otimes m}\simeq \mathcal{O}_{N_0}$. Then $\oR_{g,m}$ is a connected component of  $\M_g^{1/m}$ corresponding to pairs $(N_0,L)$ such that: $L^{\otimes d}\neq \O_{N_0}$ for all $d | m$ and $d\neq m$. We recall that the Torelli map is generically an immersion. Therefore the composition of maps from $\oR_{g,m}$ to $A_{g,(1,m,\ldots,m)}$
$$
(N_0,L) \mapsto (JN_0, L) \mapsto JN_0/\langle L\rangle=\pi^*JN_0
$$
conserve the dimensions. The dual map $A\to \hat{A}$ from $A_{g,(1,m,\ldots,m)}$ to $A_{g,(1,\ldots,1,m)}$ is \'etale. Therefore the map which associates $JN/P$ to $(N_0,L)\in \oR_{g,m}$ conserves the dimension. Finally using Proposition\ref{dimquot}, the map which associates to $(X,K)$ the ppav $X/K$ where $K$ is a m.t.i. subgroup of $X_m$ is \'etale. 

Therefore the locus of abelian varieties with an $m$-minimal curve of genus $mg-m+1$ is at least $3g-3$.
\end{proof}

\section{Quotient of Prym varieties}

The last category of examples has been extensively studied. Let $(N_0\to N)\in \oR_{g,n,r}$ be a map of degree $m$ then the Prym variety associated to this map is equal to $(\ker\Nm)_0$. Let $K$ be a m.t.i. subgroup of the kernel of $\ker(\mu_P)$ where $\mu_P$ is the induced polarization on the dual of $P$. Then the quotient $\hat{P}/K$ is $m$-minimal. In \cite{LanOrt} the authors proved that the map $\oR_{g,n,r}\to A_{m(g-1)+1+r/2,\delta}$ is generically finite. Because the dual map and the quotient by m.t.i. conserve the dimensions we have that the locus obtained by this construction is of the same dimension as $\oR_{g,n,r}$.

In particular we can easily check that the dimension of the locus of $m$-minimal abelian varieties in $A_g$ that we obtain with this construction is of dimension at most $2(g-1+m)-3$. Moreover we can check that the genus of the curve in the $m$-minimal class has genus at most $2g+1$.

\end{document}